\documentclass[smallextended,times]{article} 
\usepackage{graphicx}

\usepackage{mathptmx}      

\usepackage{amssymb}
\usepackage{amsmath}
\usepackage{amscd}
\usepackage{amsthm}
\usepackage[all]{xy}
\usepackage{manfnt}

\usepackage{enumitem}

\usepackage{centernot}

\usepackage{bm}

\def \Br {{\rm{Br}}}
\def\ov{\overline}
\def \Gal {{\rm{Gal}}}
\def \Pic {{\rm {Pic}}}
\def\NS{{\rm NS}}
\def \Z {{\mathbb Z}}
\def \Q {{\mathbb Q}}
\def\un{\underline}
\def \Aut{{\rm Aut}}
\def\inv{{\rm inv}}
\def \A{{\bm A}}

\def \Spec {{\rm{Spec\,}}}

\usepackage{tikz-cd}

\newtheorem{theo}{Theorem}[section]

\newtheorem{lem}[theo]{Lemma}

\theoremstyle{definition}

\theoremstyle{remark}

\begin{document}


%
%

\title{Applications of the fibration method to the Brauer-Manin obstruction
to the existence of zero-cycles on certain varieties.
}



\author{Evis Ieronymou }

\date{}

\maketitle

\begin{abstract}
We study the Brauer-Manin obstruction to the existence of zero-cycles of degree $d$ on certain classes of varieties over number fields.
We generalize existing results in the literature and prove some results about fibrations over the projective line, where the geometric Brauer group of the generic fibre is not assumed to be finite. The idea is to assume that the Brauer-Manin obstruction to the Hasse principle is the only one for certain fibres and then deduce analogous results for zero-cycles.
\end{abstract}



\section{Introduction}

 The theory of the Brauer-Manin obstruction to the Hasse principle and its variants has developed considerably since its introduction by Manin. In the context of zero-cycles a wide open conjecture is that the Brauer-Manin obstruction to the existence of a zero-cycle of degree $d$ is the only one \textit{for any} smooth, projective, geometrically integral variety over a number field. This conjecture was put forward in various forms and for various classes of smooth, projective,
geometrically irreducible varieties by Colliot-Th\'el\`ene and Sansuc \cite{CT.S} (see also \cite{CT95}) and by Kato and Saito \cite{KS}.
We refer the
reader to \cite{W} for more information and a more refined form of the conjecture.

In general however there is not much evidence for the aforementioned conjecture. Liang was able to prove the conjecture for rationally connected varieties by assuming the corresponding conjecture
concerning rational points \cite{Lia}, see also his subsequent work \cite{Lia2}, \cite{Lia3}, \cite{Lia4}, \cite{Lia5}. Following the ideas in \cite{Lia}, the results were extended to $K3$ surfaces \cite{Ier1} and to Kummer varieties associated to $2$-coverings of abelian varieties \cite{BalN}.

In this short note we  refine the arguments appearing in \cite{Ier1} by using some recent results of Cadoret and Charles about uniform boundedness of Brauer groups \cite{CC}, in order to generalize the main results of \cite{Ier1} and \cite{BalN}.
In particular, besides having somewhat wider applicability our results allow us to remove the restriction that $\delta$ is odd that appears in \cite[Thm. 1.1]{BalN}. 
We also present two results about families over $\mathbb P^1$, whose proof is an application of the results in \cite{HW} on the fibration method for zero-cycles.
We remark that in contrast to Liang's results, we do not assume that the geometric Brauer group of the generic fibre of the corresponding fibration used is finite.

For the rest of this note $X/k$ denotes a smooth, projective and geometrically integral variety over a number field $k$.
When $G$ is an abelian group and $n$ is a positive integer, we denote by $G[n]$ the subgroup of $G$ consisting of the elements which are killed by $n$ and by $G\{n\}$ the subgroup of $G$ consisting of the elements which are killed by some power of $n$.

We say that $X$ satisfies $\textrm{BM}_n$ if the following implication holds.
$$
X(\mathbb A_k)^{\Br(X)}=\emptyset \Rightarrow X(\mathbb A_k)^{\Br(X)\{n\}}=\emptyset.
$$
This terminology was introduced in \cite{CV}, where they proved among other things that Kummer varieties associated to $2$-coverings of abelian varieties satisfy $\textrm{BM}_2$ \cite[Thm. 1.7]{CV}.

Let $L_n$ be the following statement.

{\it The Brauer-Manin obstruction to the Hasse principle attached to $\Br(X)\{n\}$ is the only one.}

In other words, $L_n$ is the statement that the following implication holds.
$$
X(\mathbb A_k)^{\Br(X)\{n\}}\neq \emptyset \Rightarrow X(k)\neq\emptyset.
$$

Trivially, if $\textrm{BM}_n$ holds then $L_n$ is equivalent to the statement that the Brauer-Manin obstruction to the Hasse principle is the only one.

Finally, recall that the Brauer-Manin obstruction to the existence of a zero-cycle of degree $d$ on $X$ is the only one means that
if there exists a family of local zero-cycles of degree $d$ on $X$ which is orthogonal to $\Br(X)$ then there exists a zero-cycle of degree $d$ on $X$.

Our first result is the following:

\begin{theo}\label{OneX}
Let $X/k$ be a smooth, projective and geometrically integral variety over a number field $k$.
 Suppose the following.
\begin{itemize}
\item[(1)] $\Br(X)/\Br_0(X)$ and $\Br(\ov X)^{\Gal(\ov k/k)}$ are finite abelian groups;
  \item[(2)] $\Pic(\ov X)$ is finitely generated and torsion free, equivalently $H^1(X,\mathcal O_X)$ is trivial and $\NS(\ov{X})$ is torsion-free;
  \item[(3)] One of the following holds.

  \begin{itemize}

  \item[(I)]  The integral Mumford-Tate conjecture holds for $X$ and the Brauer-Manin obstruction to the Hasse principle is the only one for $X_F$ for all number fields $F$;

    \item[(II)]           \begin{itemize}
                                             \item[(a)] $X_F$ satisfies $L_n$ for all number fields $F$ and some positive integer $n$ and
                                            \item[(b)] $X$ satisfies the $\ell$-adic Tate conjecture for divisors, for all $\ell$ dividing $n$.
                       \end{itemize}

 \end{itemize}

                               \end{itemize}
Then the Brauer-Manin obstruction to the existence of a zero-cycle of degree $d$ on $X$ is the only one, for any integer $d$.

If instead of 3 only 3 (II) (a) holds then  the Brauer-Manin obstruction to the existence of a zero-cycle of degree $d$ on $X$ is the only one, for any integer $d$ coprime to $n$.

\end{theo}
{\it Remarks}
  \begin{enumerate}
    \item The integral Mumford-Tate conjecture for divisors holds for $K3$ surfaces (see e.g. \cite[Cor. 5.3]{OS}).
     The $\ell$-adic Tate conjecture for divisors holds for abelian varieties, $K3$ surfaces and varieties dominated by products of varieties that satisfy it.
    It also holds for Kummer varieties attached to 2-coverings of abelian varieties.
    \item Suppose that $X$ is a $K3$ surface. By \cite[Thm. 1.2]{SZ} assumption (1) of the Theorem holds. Applying the Theorem (with assumption 3 (I)), we recover the main result of \cite{Ier1}.
    \item  Suppose that $X$ is a  Kummer variety attached to a 2-covering of an abelian variety.  By \cite[Cor. 2.8]{SZ2} assumption (1) of Theorem \ref{OneX} holds. Applying
    Theorem \ref{OneX} (with assumption 3 (II) for $n=2$), we recover the main result of \cite{BalN}, without the restriction that the degree of the zero-cycle is odd.

  \end{enumerate}

The next two results concern families over $\mathbb P^1$. Their proof is based on the fibration method for zero-cycles as elaborated in \cite{HW}.
The fact that there is no finiteness assumption on the geometric Brauer group of the generic fibre makes the proofs a
 bit different than the proofs of similar results in the literature (cf. \cite[Thm 2.5, Rem. 2.8 and Lem. 3.10]{Lia5}).
\begin{theo}\label{Fam.P1}
 Let $X/k$ be a smooth, projective and geometrically integral variety over a number field $k$. Let $f:X\to \mathbb P^1$ be a smooth and proper morphism. Let $n$ be a positive integer and denote by $\eta$ the generic point of $\mathbb P_k^1$.
Assume the following.
\begin{enumerate}
\item $\Br(X)/\Br_0(X)$ is finite and $\Br(X)\to \Br(X_{\eta})/\Br_0(X_{\eta})$ is surjective;
\item $\Pic(X_{\ov{\eta}})$ is finitely generated and torsion free;
\item $X_{\eta}$ satisfies the $\ell$-adic Tate conjecture for divisors, for any $\ell$ dividing $n$;
\item There exists a Hilbert subset $M\subseteq \mathbb P^1$ such that for any closed point $c\in M$, the variety  $X_c$ satisfies $L_n$.

\end{enumerate}
Then the Brauer-Manin obstruction to the existence of a zero-cycle of degree $d$ on $X$ is the only one, for any integer $d$.
\end{theo}

We have a similar result with slightly different assumptions.

\begin{theo}\label{Fam.P1b}
 Let $X/k$ be a smooth, projective and geometrically integral variety over a number field $k$.
  Let $f:X\to \mathbb P^1$ be a smooth and proper morphism. Let $n$ be a positive integer and denote by $\eta$ the generic point of $\mathbb P_k^1$.
Assume the following.
\begin{enumerate}
\item $\Br(X)/\Br_0(X)$ is finite and $\Br(X)\to \Br(X_{\eta})/\Br_0(X_{\eta})$ is surjective;
\item  $\Pic(X_{\ov{\eta}})$ is finitely generated and torsion free;
\item There exists a Hilbert subset $M\subseteq \mathbb P^1$ such that
\begin{enumerate}
  \item for any closed point $c\in M$ the integral Mumford-Tate conjecture holds for $X_c$ and the Brauer-Manin obstruction to the Hasse principle is the only one for  $X_c$ and
  \item for any positive integer $C$ the set $\{X_{\ov c}: c\in M, \ \ [k(c):k]=C\}$, contains finitely many isomorphism classes of varieties over $\ov k$.

\end{enumerate}

\end{enumerate}
Then the Brauer-Manin obstruction to the existence of a zero-cycle of degree $d$ on $X$ is the only one, for any integer $d$.
\end{theo}

\section{Proof of Theorem \ref{OneX}}

We will follow the proof of \cite[Thm. 1.2]{Ier1} adjusting it accordingly. Before giving the details, let us give the key ideas of the proof of \cite[Thm. 1.2]{Ier1}. First we note that instead of $X$ it suffices to prove the analogous statement for $X\times \mathbb P^1$ (see eg. \cite[Lem. 3.1]{Ier1}).
Let us work with the more general situation where we have a smooth proper map $g:Y\to \mathbb P^1$.
Given a finite subset $B\subseteq \Br(Y)$, a family of local zero-cycles of degree $d$ on $Y$ orthogonal to $B$ and a Hilbert subset $H$ of $\mathbb P_k^1$, we can use the fibration method in order to find a closed point $s$ in $H$
with the property that the fibre $Y_{s}$ has a family of local points orthogonal to the image of $B$ in $\Br(Y_s)$. Moreover we have some control on the degree of the field extension $k(s)/k$.
If the Brauer-Manin obstruction to the Hasse principle is the only one for  $Y_s$ and $B$ surjects onto $\Br(Y_s)/\Br_0(Y_s)$ then we know that there exists a $k(s)$-rational point on $Y_s$, which in turn will allow us
to construct a global zero-cycle on $Y$ of the appropriate degree. The main difficulty in the above procedure is to define $H$ in a way which will guarantee that $B$ surjects onto $\Br(Y_s)/\Br_0(Y_s)$ for
the closed point $s\in H$ given to us by the fibration method.

 From now on we use the notation and terminology used in the proof of \cite[Thm. 1.2]{Ier1}. In particular $\Gamma_k$ denotes the Galois group $\Gal(\ov k/k)$.

\un{ Assume 3(I).} Then the proof is exactly the same as the proof of \cite[Thm. 1.2]{Ier1}: we need only replace the reference \cite[Thm. C]{OS} to the reference \cite[Thm. 5.1]{OS}.
Note that for any integer $C$, $\Br(\ov{X})$ has finitely many elements of order $\leq C$ (see e.g. \cite[Cor. 5.2.8]{CT.Sk}).

\un{Assume 3(II).} In this case in the proof of \cite[Thm. 1.2]{Ier1} we need to define the Hilbert subset, $H$, of $\mathbb P^1$ differently. According to \cite[Thm 1.2.1]{CC} there exists a constant $C$ such that
 $$|\Br(\ov X)^{\Gal(\ov k/N)}\{n\}| \leq C$$
 for any field extension $N/k$ with $[N:k]\leq \delta$. Now, $\Br(\ov X)$ contains finitely many elements of  order at most $C$ and each one of them is stabilized by a subgroup of $\Gamma_k$ of finite index, which corresponds to a finite extension
of $k$. Taking the Galois closure of the finitely many such extensions we produce a finite Galois extension $M/k$. We now continue just like in the proof of \cite[Thm. 1.2]{Ier1} with the definition of $H$ and the application of \cite[Thm. 6.2]{HW}.

In order to conclude in this case we  need to show that the image of $\Br(X)$ in $\Br(X_{k'})/\Br_0(X_{k'})$ contains the image of $\Br(X_{k'})\{n\}$, where $k'=k(P)$ and $P$ is a closed point of $H$ with $[k':k]=\delta$.
By our constructions we have the following commutative diagram with exact rows
(this is similar to the diagram appearing in in the proof of \cite[Thm. 1.2]{Ier1}, the difference being that in this case we consider the $n$-primary part of the various groups).

\[
\begin{tikzcd}[column sep=tiny]
0\arrow{r}&H^1(\Gal(F/k), \Pic \ov X)\{n\}\arrow{r}\arrow{d} &\Br(X)/\Br_0(X)\{n\}\arrow{r}\arrow{d}
&\Br( \ov X)^{\Gamma_k}\{n\}\arrow{d} \arrow{r}& H^2(\Gal(F/k), \Pic(\ov X))\{n\}\arrow{d}\\
0\arrow{r}&H^1(\Gal(F'/k'), \Pic \ov X)\{n\}\arrow{r}&\Br(X_{k'})/\Br_0(X_{k'})\{n\}\arrow{r}&\Br( \ov X)^{\Gamma_{k'}}\{n\}\arrow{r}&H^2(\Gal(F'/k'), \Pic(\ov X))\{n\}
\end{tikzcd}
\]

\noindent  where the first, third and fourth vertical arrows are isomorphisms. The result follows from this.

\un{Assume 3(II) (a).} Note that it follows from the proof of \cite[Prop. 7.5]{HW} that we can assume that  the $\delta$ appearing in the proof of \cite[Thm. 1.2]{Ier1} satisfies $\delta\equiv d\mod n$.
Since $d$ is coprime to $n$ in this case, we may therefore assume that $\delta$ is coprime to $n$ as well.

In this case we also need to define the Hilbert subset, $H$, of $\mathbb P^1$ differently. Let $M/k$  be the finite Galois extension defined as the the fixed field of the kernel of the map
$\Gal(\ov k/k)\to \Aut(\Br(\ov X)[n]$.  We now continue just like in the proof of \cite[Thm. 1.2]{Ier1} with the definition of $H$ and the application of \cite[Thm. 6.2]{HW}.
We can conclude by the same arguments as the previous case, once we show that the map $\Br( \ov X)^{\Gamma_k}\{n\}\to \Br( \ov X)^{\Gamma_{k'}}\{n\}$ is an isomorphism.
It suffices to show that $\Gamma_k$ and $\Gamma_{k'}$ have the same image in $\Aut(\Br(\ov X)[\ell^m])$ for any $\ell$ dividing $n$ and any integer $m$. Fix $\ell$ dividing $n$. Since $k'$ is linearly disjoint from $M$, it follows that $\Gamma_k$ and $\Gamma_{k'}$ have the same image in $\Aut(\Br(\ov X)[n])$ and so a fortiori in  $\Aut(\Br(\ov X)[\ell])$. We can now proceed as in the last part of the proof of \cite[Lem. 4.2]{BalN},
once we justify that $\ker( \Aut ( \Br(\ov X)[\ell^m])\to \Aut ( \Br(\ov X)[\ell]))$ is an $\ell$-group. 
Since $\Br(\ov X)[\ell^m]$ is finite, the statement we need holds by Lemma \ref{grpthr}.

\begin{lem}\label{grpthr}
Let G be a finite abelian $\ell$-group, and denote by $\alpha$ the natural map $\alpha :\Aut( G )\to \Aut ( G[\ell] )$.
Then $\ker(\alpha)$ is an $\ell$-group.
\end{lem}
\begin{proof}
  This follows from \cite[Lem. 3.1 and Thm. 3.2]{M}.
\end{proof}

\section{Proof of Theorem \ref{Fam.P1}}

We will again follow the ideas of the proof of \cite[Thm. 1.2]{Ier1} adjusting them accordingly. Note that in contrast to \cite{Ier1} we are not using the trivial fibration here.

For the convenience of the reader we repeat some details. We first find a finite set $S\subset \Omega$ containing the infinite
places so that $X$ has a smooth projective model $\mathcal X$ over $\mathcal
O_S$. Let $\alpha_i\in \Br(X)$, $1\leq i \leq n$, be elements whose classes generate
$\Br(X)/\Br_0(X)$. By enlarging $S$ if necessary, we can furthermore assume
that each $\alpha_i$ extends to an element of $\Br(\mathcal X)$. Hence for any
$1\leq i \leq n$, $v \notin S$ and $b\in Z_0(X_{k_v})$ we have that
$\inv_v(\langle \alpha_i,b \rangle)=0$. Let $B$ be the span of
the $\alpha_i$ in $\Br(X)$.

Let $z_{\A}=\{z_v\}_{v \in \Omega}$ be a family of local zero-cycles of
degree $d$ on $X$, which is orthogonal to $B$. In a similar way as in the proof of \cite[Thm. 1.2]{Ier1}, we use results from \cite{HW} in order to construct a zero-cycle, $z$, of degree $d-\delta$ on $X$ and  a family of effective, reduced local zero-cycles of constant
degree $\delta$ on $X$, say $z'_{\A}$, such that $z'_{\A}$ is orthogonal to $B$.

We now define a Hilbert subset of $\mathbb P^1_k$ as follows.
By \cite[Fact 3.4.1]{CC} the image of $(\mathbb P^1)^{\text{ex}}_{\ell,\leq \delta}$ in $\mathbb P^1$ is finite (for the notation see \cite[\textsection 3]{CC}). Denote by $U_{\ell}$ its complement and set $H_1:=\cap_{\ell \in R}U_{\ell}$, where $R$ is the set of primes dividing $n$. Let $s\in H_1$ with $[k(s):k]=\delta$. By \cite[Thm. 1.2.1 and Prop. 3.2.1]{CC} there exists an integer $C$, which only depends on $n$ and $\delta$, such that
$$
|\Br(X_{\ov s})^{\Gal(\ov k/k(s))}\{n\}|<C.
$$
Moreover the natural specialisation maps $\NS(X_{\ov{\eta}})\to \NS(X_{\ov{s}})$, and $\Br(X_{\ov{\eta}})\to \Br(X_{\ov{s}})$ are isomorphisms.

Denote by $k(t)$ the function field of $\mathbb P^1_k$, so that $\Spec(k(\eta))=\Spec(k(t))$. Note that $\Gal(\ov{k(t)}/\ov{k}(t))$ acts trivially on $H^2(X_{\ov{\eta}},\Q/\Z(1))$; this follows from the fact that $H^2(X_{\ov{\eta}},\Q/\Z(1))$
is in a natural way a $\pi_1(\mathbb P^1_k,\ov{\eta})$-module (cf. \cite[proof of Lem. 4.3]{HW} for the last statement).
By \cite[II, Thm. 3.1]{Gro} it follows that $\Br(X_{\ov{\eta}})$ is a $\Gamma_k$-module.
Consider the finitely many elements of $\Br(X_{\ov \eta})$ whose order is less  than $C$.
Each one of them is stabilized by a subgroup of $\Gal(\ov k/k)$ of finite index, which corresponds to a finite extension
of $k$. Taking the Galois closure of the finitely many such extensions we produce a finite Galois extension $F/k$.
For this $F$ we now choose a Hilbert subset $H_2$ of $\mathbb P^1_k$ as in \cite[Lem. 3.2]{Ier1}.

For the remainder of the proof we denote by $s$ a closed point of $H_1\cap H_2$ with $[k(s):k]=\delta$,
and we freely identify $s$ and $\eta$ with $\Spec(k(s))$ and $\Spec(k(\eta))$. Sometimes we will denote the latter by $\Spec(k(t))$.

Recall that $H^3(k,\mathbb G_m)=H^3(k(t),\mathbb G_m)=0$ for a number field $k$ (see e.g. \cite[Lem. 2.6]{CTP}). It follows from the Hochschild-Serre spectral sequence that

$$
0\to H^1(\eta, \Pic(X_{\ov{\eta}})) \to\Br(X_{\eta})/\Br_0(X_{\eta}) \to \Br(X_{\ov{\eta}})^{\Gamma_{\eta}}\to H^2(\eta, \Pic(X_{\ov{\eta}}))
$$
and

$$
0\to H^1(s, \Pic(X_{\ov{s}})) \to\Br(X_{s})/\Br_0(X_{s}) \to\Br(X_{\ov{s}})^{\Gamma_{s}}\to H^2(s, \Pic(X_{\ov{s}}))
$$
are exact (this is well-known see e.g. \cite[Lem. 3.6]{Lia5}).

Because of our assumptions we have that  $\Br(X_{\ov{\eta}})^{\Gamma_{\eta}}= \Br(X_{\ov{\eta}})^{\Gamma_{k}}\subseteq \Br(X_{\ov{s}})^{\Gamma_{s}}$.
As $\Br(X_{\ov{s}})^{\Gamma_{s}}\{n\}$ is finite the same is true for its subgroup  $\Br(X_{\ov{\eta}})^{\Gamma_{\eta}}\{n\}$. Since moreover $\Pic(X_{\ov{\eta}})$ is finitely generated and free, we can choose a large enough finite Galois extension $L/k(t)$ such that  the image of
$\Br(X_{\ov{\eta}})^{\Gamma_{\eta}}\{n\}$ lies in the subgroup $H^2(\Gal(L/k(t)),\Pic(X_{\ov{\eta}}) )\{n\}$ of  $H^2(\eta,\Pic(X_{\ov{\eta}}) )\{n\}$(see \cite[proof of Lem. 3.3]{Ier1}).

Let $H_3$ be the Hilbert subset corresponding to the finite morphism whose generic fibre corresponds to the finite extension $L/k(t)$.
Set $H:=H_1\cap H_2\cap H_3\cap M$ and assume additionally that $s$ a closed point of $H$ with $[k(s):k]=\delta$.
As explained in  \cite[beginning of pg. 231, proof of Lem. 3.10]{Lia5}, we have that $\Gal(L/k(t))\cong \Gal(l/k(s))$ where $l$ is the specialisation of $L$ to $k(s)$ which is a field.

We now apply \cite[Thm 6.2]{HW} and continue just like in the proof of \cite[Thm. 1.2]{Ier1}.
In order to conclude we need to show that the image of $B$ in $\Br(X_{s})/\Br_0(X_{s})$ contains the image of $\Br(X_{s})\{n\}$.

Note that because of assumption (1) we have the following commutative diagram
\[
\begin{tikzcd}[column sep=tiny]
B \arrow{r}&\Br(X)/\Br_0(X) \arrow{r}\arrow{rd}&\Br(X_{\eta})/\Br_0(X_{\eta})\arrow{d} \\
        &                           &\Br(X_{s})/\Br_0(X_{s}) \\
\end{tikzcd}
\]
where the map $B\to \Br(X_{\eta})/\Br_0(X_{\eta})$ is surjective (for the definition of the map $\Br(X_{\eta})/\Br_0(X_{\eta})\to \Br(X_{s})/\Br_0(X_{s})$
under our assumptions see \cite[\textsection 3.3, pg 237]{Har}).

From all the above we deduce the existence of the following commutative diagram with exact rows.

\[
\begin{tikzcd}[column sep=tiny]
\Br(X_{\eta})/\Br_0(X_{\eta}) \{n\}\arrow{d}\arrow{r}{q_{\eta}}&\Br(X_{\ov{\eta}})^{\Gamma_{\eta}}\{n\}\arrow{d}\arrow{r} &H^2(\Gal(L/k(t)), \Pic(X_{\ov{\eta}}))\{n\}\arrow{d} \\
        \Br(X_{s})/\Br_0(X_{s})\{n\}\arrow{r}{q_{s}} &\Br(X_{\ov{s}})^{\Gamma_{s}}\arrow{r}\{n\} &H^2(\Gal(l/k(s)), \Pic(X_{\ov{s}}))\{n\} \\
\end{tikzcd}
\]

Like in the proof of \cite[Thm. 1.2]{Ier1} it follows from our constructions that the two right-most vertical maps are isomorphisms.
Moreover, again by our constructions, we have that $\ker(q_s)$ is canonically isomorphic to $\ker(q_{\eta})$ which is isomorphic to $H^1(\Gal(L/k(t)), \Pic(X_{\ov{\eta}}))\{n\}$.
Since $B$ surjects onto $\Br(X_{\eta})/\Br_0(X_{\eta})$ it follows by a diagram chase that $B\{n\}\to \Br(X_{s})/\Br_0(X_{s})\{n\}$ is surjective, which concludes the proof.

  \section{Proof of Theorem \ref{Fam.P1b}}

  This is very similar to the proof of Theorem \ref{Fam.P1}. We note the main difference. In this case we cannot appeal to \cite[Thm. 1.2.1]{CC} to show that
$$
|\Br(X_{\ov s})^{\Gal(\ov k/k(s))}\{n\}|<C.
$$
Instead, we need to appeal to  \cite[Thm. 5.1]{OS}, in order to show the stronger statement that there exists an integer $C$ such that
$$
|\Br(X_{\ov s})^{\Gal(\ov k/k(s))}|<C.
$$
 The rest of the proof is quite similar to the proof of Theorem \ref{Fam.P1}.

\section* {Acknowledgments}

The author would like to thank the referee for a careful reading of the manuscript and several useful comments.

\noindent Department of Mathematics and Statistics, University of Cyprus, P.O. Box 20537,
1678, Nicosia, Cyprus.

\noindent ieronymou.evis@ucy.ac.cy

\end{document}